\documentclass[reqno,12pt]{amsart}
\usepackage{geometry}
\usepackage{amsmath,amssymb,mathrsfs,color}
\usepackage[upint]{stix}
\usepackage[colorlinks,
linkcolor=red,
anchorcolor=green,
citecolor=blue,
]{hyperref}
\usepackage{subfigure}
\usepackage{float}
\usepackage{times}
\usepackage{tikz}
\usetikzlibrary{intersections}
\usepackage{paralist}

\usepackage{tikz}
\usetikzlibrary{intersections}
\usepackage{paralist}

\makeatletter
\def\setliststart#1{\setcounter{\@listctr}{#1}%
  \addtocounter{\@listctr}{-1}}
\makeatother

\makeatletter
\@addtoreset{figure}{section}
\makeatother

\usepackage{calc}
\newtheorem{The}{Theorem}[section]
 \newtheorem{Cor}[The]{Corollary}
 \newtheorem{Lem}[The]{Lemma}
 \newtheorem{Pro}[The]{Proposition}
 \theoremstyle{definition}
 \newtheorem{defn}[The]{Definition}
 \theoremstyle{remark}
 \newtheorem{Rem}[The]{Remark}
 
 \numberwithin{equation}{section}

\newcommand{\R}{\mathbb{R}}

\newcommand{\N}{\mathbb{N}}

\title[Global propagation of singularities]{Global propagation of singularities for  discounted Hamilton-Jacobi equations}
\author{Cui Chen \and Jiahui Hong \and Kai Zhao}
\address{School of Mathematical Sciences, Jiangsu University, Zhenjiang 212013, China}
\email{chenc@ujs.edu.cn}
\address{Department of Mathematics, Nanjing University, Nanjing 210093, China}
\email{hjh9413@163.com}
\address{School of Mathematical Sciences, Fudan University, Shanghai 200433, China}
\email{zhao$\_$kai@fudan.edu.cn}
\date{\today}
\begin{document}
\maketitle
\begin{abstract}
 The main purpose of this paper is to study the global propagation of singularities of viscosity solution to discounted Hamilton-Jacobi equation
\begin{equation}\label{eq:discount 1}\tag{HJ$_\lambda$}
	\lambda v(x)+H( x, Dv(x) )=0 , \quad x\in \R^n.
\end{equation}
We reduce the problem for equation \eqref{eq:discount 1} into that for a time-dependent evolutionary Hamilton-Jacobi equation. We proved that the singularities of the viscosity solution of \eqref{eq:discount 1} propagate along locally Lipschitz singular characteristics which can extend to $+\infty$. We also obtained the homotopy equivalence between the singular set and the complement of associated the Aubry set with respect to the viscosity solution of equation \eqref{eq:discount 1}.
\end{abstract} 

\section{introduction}
It is commonly accepted that, in optimal control, a crucial role is played by the Hamilton-Jacobi equation
\begin{equation}\label{eq:HJ}\tag{HJ$_e$}
	\begin{cases}
	D_t u(t ,x)+H(t,x,D_x u(t,x) )=0 \quad &(t,x)\in \R^+ \times \R ^n, \\
	u(0,x)=u_0(x) &x\in \R^n.
\end{cases}
\end{equation}
It is well known that the singularities of such solutions propagate locally along generalized characteristics. The evidence of irreversibility for the Hamilton-Jacobi equation  is the propagation  of  singularities. Once a singularity is created, it will propagate forward in time up to $+\infty$.  For a comprehensive survey of this topic, the readers can refer to \cite{Cannarsa_Cheng2021}.
 
The theory of local propagation of the singularities of the viscosity solutions of \eqref{eq:HJ} has been established in \cite{Albano_Cannarsa2002} by introducing the notion of generalized characteristics (see also \cite{Cannarsa_Yu2009}, \cite{Yu2006}). Other progress on the local propagation includes the strict singular characteristics (\cite{Khanin_Sobolevski2016}, see also \cite{Stromberg2013}). A recent remarkable result by Cannarsa and Cheng established the relation between generalized characteristics and strict singular characteristics on $\R^2$ (\cite{Cannarsa_Cheng2020}). 

In the paper \cite{Cannarsa_Cheng3}, Cannarsa and Cheng introduced an intrinsic method and obtained a global propagation result for time-independent Hamiltonian (see also \cite{Cannarsa_Mazzola_Sinestrari2015}). By a  procedure of sup-convolution with the kernel the fundamental solutions of associated autonomous  Hamilton-Jacobi equations, they constructed a global singular  arc $\mathbf{y}_x(t):[0,t_0]\to\R^n$ from an initial singular point $x$ and $t_0$ is independent of the initial point $x$. The uniformness of such $t_0$ holds because of  uniform conditions (L1)-(L3) in \cite{Cannarsa_Cheng3} . In \cite{Cannarsa_Cheng2018}, they ask the following problem $\mathbf{A3}$:
\begin{align*}
	\text{Can we drop the uniformness requirement of such } t_0 \text{ to obtain a global result ?}
\end{align*}
The first task of this paper is to drop the uniformness of $t_0$ which can not be guaranteed by, e.g., the so called Fathi-Maderna conditions (\cite{Fathi_Maderna2007}) which we will use for our purpose. In this paper, we showed that the answer to problem $\mathbf{A3}$ is affirmative for time-dependent case and discounted case.

There is a very natural connection between the discounted Hamilton-Jacobi equation \eqref{eq:discount 1} and the evolutionary Hamilton-Jacobi equation \eqref{eq:HJ} using a conformal Hamiltonian (see, for instance, \cite{Maro_Sorrentino2017}) or a contact Hamiltonian (see, for instance, \cite{Chen_Cheng_Zhang2018}). More precisely, if $v$ is the unique viscosity solution of \eqref{eq:discount 1}, we define
\begin{align*}
	u(t,x)=e^{\lambda t}v(x),\qquad x\in\R^n, t>0.
\end{align*}
Then $u(t,x)$ is a viscosity solution of \eqref{eq:HJ} with a time-dependent Hamiltonian in the form $e^{\lambda t}H(x,p/e^{\lambda t})$. Notice that $v$ and $u$ share the singularity. Thus, we can discuss the problem of propagation of singularities for equation \eqref{eq:HJ} instead of equation \eqref{eq:discount 1}. We developed the intrinsic method in \cite{Cannarsa_Cheng3} adapt to our problem which has more technical difficulty comparing to the time-independent case (\cite{Cannarsa_Cheng_Fathi2019}). 

Now we introduce the associated Lagrangian as
\begin{align*}
	L(s,x,v)=\sup_{p\in\R^n}\{p\cdot v-H(s,x,p)\},\qquad s>0, x\in\R^n, v\in\R^n.
\end{align*}
To deal with evolutionary Hamilton-Jacobi equation \eqref{eq:HJ}, we suppose $L=L(s,x,v):\R\times \R^n\times\R^n\to\R$ is of class $C^2$ and satisfies the following assumptions:
\begin{enumerate}[(L1)]
    \item $L(s,x,\cdot)$  is strict convex on $\R^n$ for all $s\in[0,\infty ] $, $x\in \R^n $ .
    \item For any fixed $T>0$, there exist $c_{T} > 0$ and two superlinear and nondecreasing function $\overline \theta_T, \theta_T:[0,+\infty)\to [0,+\infty)$, such that
	$\overline \theta_T(|v|)  \geqslant L(s, x, v)\geqslant \theta_T(|v|)-c_{T}, $ for all $ (s,x,v)\in [0,T] \times \R^n \times \R^n. $		
	\item There exists $\widetilde C_1,\widetilde C_2:\R \rightarrow [0,+\infty) $ such that $|L_t(s,x,v)|\leqslant \widetilde C_1(T)+\widetilde C_2(T)L(s,x,v)$ for all $(s,x,v)\in [0,T] \times \R^n \times \R^n$.
\end{enumerate}

We say that a curve $\gamma:[a,b]\to \R^n$ is $\lambda$-calibrated curve for equation \eqref{eq:discount 1} if
\begin{equation}\label{eq:calibrated 1}
	e^{\lambda b}u(\gamma(b)) - 	e^{\lambda a}u(\gamma(a)) =\int_a^b e^{\lambda t}L(\gamma(t),\dot \gamma(t)) \ dt,
\end{equation} 
and a curve $\gamma:[a,b]\to \R^n$ is calibrated curve for equation \eqref{cauchy equation} if
\begin{equation}\label{eq:calibrated 1}
	u(b,\gamma(b)) - u(a,\gamma(a)) =\int_a^b L(t,\gamma(t),\dot \gamma(t)) \ dt.
\end{equation}
A point $x \in \R^n $ is a cut point of $u$ if no backward $\lambda$-calibrated curve of equation \eqref{eq:discount 1} with Hamilton $H$ ending at $x$ can be extended beyond $x$. A point $(t,x) \in \R^+ \times \R^n $ is a cut point of $u$ if no backward calibrated curve of equation \eqref{cauchy equation} with Hamilton $H$ ending at $(t,x)$ can be extended beyond $(t,x)$. In both cases, we denote by $\mbox{\rm{Cut}}\,(u)$ the set of cut points of $u$. If $u$ is a viscosity solution of \eqref{eq:discount 1} or \eqref{cauchy equation}, a singularity of $u$ is a point where $u$ is not differentiable. We denote by $\mbox{\rm{Sing}}(u)$ the set of singularities of $u$. It is well known that $\mbox{\rm{Sing}}\,(u) \subset \mbox{\rm{Cut}}\,(u) \subset \overline{\mbox{\rm{Sing}}\,(u)}$. 	

Our main result for the time-dependent case is: Let $L$ be a Lagrangian which satisfies \mbox{\rm{(L1)-(L3)}} and let $H$ be the associated Hamiltonian. Suppose $u_0=u(0,\cdot):\R^n \rightarrow \R $ is a Lipschitz continuous  function. Then for any fixed $ (t_0,x) \in \mbox{\rm{Cut}}(u) \subset \R^+ \times \R^n $, there exists a curve $\mathbf{x}:\R \rightarrow \R^n $ with $\mathbf{x}(t_0)=x $, such that $ (s,\mathbf{x}(s)) \in \mbox{\rm{Sing}}(u) $ for all $s\in [t_0,+\infty) $.
	Moreover, If  condition \rm{\textbf{(A)}} (see Section \ref{Sect.3}) holds, then for any $T>0$, $\mathbf{x}(s)$ is a Lipschitz curve on $s\in [t_0,T]$.

Similarly, for the discounted equation \eqref{eq:discount 1} we denote by $L$ the associated Lagrangian of $H$. We suppose $L=L(x,v):\R^n\times\R^n\to\R$ is of $C^2$ class and satisfying the following assumptions:
\begin{enumerate}[\rm (L1')]
  \item $L(x,\cdot)$ is strictly convex for all $x\in\R^n$.
  \item There exist $c_1,c_2\geqslant 0$ and two superlinear functions $\theta_1,\theta_2:[0,+\infty)\to[0,+\infty)$ such that
  \begin{align*}
  \theta_2(|v|)+c_2 \geqslant L(x,v) \geqslant \theta_1(|v|)-c_1,\qquad \forall (x,v)\in\R^n\times\R^n.
  \end{align*}
\end{enumerate}

Our main result for the discounted case is: Let $L$ be a Lagrangian which satisfies \mbox{\rm{(L1')-(L2')}} and $H$ be the associated Hamiltonian and $\lambda>0$. Suppose $v:\R^n\to \R$ is  a Lipschitz continuous semiconcave viscosity solution of \eqref{eq:discount 1}. Then
\begin{enumerate}[(1)]
		\item  for any fixed $x\in\mbox{\rm{Cut}}(v)$, there exists a locally Lipschitz curve $\mathbf{x}:[0,+\infty)\rightarrow \R^n $ with $\mathbf{x}(0)=x $, such that $ \mathbf{x}(s) \in \mbox{\rm{Sing}}(v) $ for all $s\in [0,+\infty) $,
		\item  the inclusions 
	$$
	\mbox{\rm{Sing}}(v) \subset \mbox{\rm{Cut}}(v) \subset \Big( \R^n \backslash \mathcal{I}(v) \Big) \cap  \overline{\mbox{\rm{Sing}}(v)} \subset \R^n \backslash \mathcal{I}(v)
	$$
	are all homotopy equivalences and the spaces $\mbox{\rm{Sing}}\,(u)$ and  $\mbox{\rm{Cut}}\,(u)$ are all locally contractible.
\end{enumerate}

It worth noting that the construction of the homotopy equivalence here we used is very similar to what used in \cite{Cannarsa_Cheng_Fathi2017}, \cite{CCMW2019} and \cite{Cannarsa_Cheng_Fathi2019}. The general notion of the cut locus of $u$ for contact type Hamilton-Jacobi equation was studied in \cite{Cheng_Hong2021} recently for smooth initial data.

This paper is organized as follows. In Sect. \ref{Sect.2}, we introduce Lax-Oleinik operator associated to \eqref{eq:HJ} and give our global result on the propagation of singularities along local Lipschitz curves under an extra condition (A). In Sect.\ref{Sect.3}, we discuss the global propagation of singularities for discounted Hamiltonian \eqref{eq:discount 1} and give homotopy equivalence results as an application.This paper contains three appendices which include some background materials and useful conclusions. In Appendix \ref{Appendix A}, we collect some relevant regularity results with respect to the fundamental solution of \eqref{cauchy equation}. In Appendix \ref{Appendix B} and Appendix \ref{Appendix C}, we give the proof of Lemma \ref{T - T + est1}, Lemma \ref{lem:estimation2} and Lemma \ref{lem:homotopy on [0,t]}.

\medskip

\noindent\textbf{Acknowledgements.} Cui Chen is partly supported by National Natural Science Foundation of China (Grant No. 11801223, 11871267).

\section{Global propagation of singularities for time-dependent Hamiltonian}\label{Sect.2}

In this section, we will discuss the connection between sup-convolution, singularities and generalized characteristics for the following time-dependent Hamilton-Jacobi equation:
\begin{equation}\label{cauchy equation}\tag{HJ$_e$}
	\begin{cases}
	D_t u(t ,x)+H(t,x,D_x u(t,x) )=0 \quad &(t,x)\in \R^+ \times \R ^n, \\
	u(0,x)=u_0(x) &x\in \R^n.
\end{cases}
\end{equation}
Let $L(s,x,v)$ be the associated Lagrangian of $H(s,x,p)$. We assume that $L(s,x,v):[0,+\infty)\times \R^n \times \R^n \to  \R$ is a $C^2$-smooth function which satisfies the following standard assumptions:
\begin{enumerate}[(L1)]
    \item $L_{vv}(s,x,v)>0 $ for any  $(s,x,v) \in [0,+\infty )\times \R^n \times \R^n $.
    \item For any fixed $T>0$, there exist $c_{T} > 0$ and two superlinear and nondecreasing functions $\overline \theta_T, \theta_T:[0,+\infty)\to [0,+\infty)$, such that
	\begin{align*}
		  \overline \theta_T(|v|)  \geqslant L(s, x, v)\geqslant \theta_T(|v|)-c_{T}, \quad (s,x,v)\in [0,T] \times \R^n \times \R^n .
	\end{align*}		
	\item There exist two locally bounded functions $\widetilde C_1,\widetilde C_2:[0,+\infty) \rightarrow [0,+\infty) $ such that for any  $T>0$,
	$$|L_t(s,x,v)|\leqslant \widetilde C_1(T)+\widetilde C_2(T)L(s,x,v),\quad  (s,x,v)\in [0,T] \times \R^n \times \R^n.$$
\end{enumerate}
For any $0\leqslant s <t <+\infty$ and $x,y\in \R^n $, we define the fundamental solution of the Hamilton-Jacobi equation \eqref{eq:HJ} as follows:
\begin{equation}\label{action}
	A_{s,t}(x,y)=\inf_{\xi \in \Gamma^{s,t}_{x,y}} \int_s^t L(\tau, \xi(\tau) ,\dot \xi(\tau) )  d \tau.
\end{equation}
where 
$$
\Gamma^{s,t}_{x,y}=\{\xi \in W^{1,1}([0,t],\R^n):\gamma(s)=x, \gamma(t)=y \}
$$
We call $\xi \in \Gamma^{s,t}_{x,y} $ a minimizer for $A_{s,t}(x,y) $ if $A_{s,t}(x,y)=\int_s^t L(\tau,\xi(\tau),\dot \xi(\tau)  ) d \tau. $ The existence of minimizers in \eqref{action} is a well known result in Tonelli's theory (see, for instance, \cite{Clarke_Vinter1985_1}). Moreover, we have the following proposition:
\begin{Pro}\label{pro:fundamental solution}
	Suppose $L$ satisfies \mbox{\rm{(L1)-(L3)}}. Then for any $0\leqslant s<t <+\infty $ and $x,y \in\R^n $, there exists $\xi \in \Gamma^{s,t}_{x,y} $ such that $\xi$ is a minimizer for $A_{s,t}(x,y) $ and the following properties hold:
	\begin{enumerate}[\rm (1)]
		\item [(1)]$\xi$ is of class $C^2$ and satisfies 
		$$\frac{d}{ds}L_v(\tau,\xi(\tau),\dot \xi(\tau))=L_x(\tau,\xi(\tau),\dot \xi(\tau)), \quad \forall \tau \in[s,t].  
		$$
		\item [(2)] Let $p(\tau)=L_v(\tau,\xi(\tau),\dot \xi(\tau))$ for $\tau \in [s,t]$. Then $(\xi,p)$ satisfies
		\begin{equation}\label{eq:H}
		\begin{cases}
			\dot \xi(\tau)=H_p(\tau,\xi(\tau),p(\tau) ), \\
			\dot p(\tau)=-H_x(\tau,\xi(\tau),p(\tau) ),
		\end{cases}\qquad \forall \tau\in[s,t].
		\end{equation}
	\end{enumerate}
\end{Pro}
In Appendix \ref{Appendix A}, we collect some relevant regularity results with respect to the fundamental solution $A_{s,t}(x,y)$. The proofs of these regularity results are similar to those in \cite{Cannarsa_Cheng3} for autonomous case.

\subsection{Semiconcave functions}
Let $\Omega \subset \R^n$ be a convex open set. We recall that a function $u:\Omega \to \R$ is said to be semiconcave (with linear modulus) if there exists a constant $C>0$ such that
$$
\lambda u(x)+(1-\lambda)u(y)-u(\lambda x+(1-\lambda )y )\leqslant \frac{C}{2}\lambda (1-\lambda )|x-y|^2,\quad \forall x,y\in \Omega,\lambda\in [0,1].
$$
For any continuous function $u:\R^n \to \R$ and $x\in \R^n$, we denote
\begin{align*}
	D^- u(x)= \Big\{ p\in T^*_x M: \lim \inf_{y\to x} \frac{u(y)-u(x)-\langle p,y-x \rangle }{|y-x|}\geqslant 0 \Big\},\\
	D^+ u(x)= \Big\{ p\in T^*_x M: \lim \sup_{y\to x} \frac{u(y)-u(x)-\langle p,y-x \rangle }{|y-x|}\leqslant 0 \Big\},
\end{align*}   
which are called the \textit{subdifferential and superdifferential} of $u$ at $x$, respectively. Let now $u:\R^n\to \R$ be locally Lipschitz and $x\in\R^n$. We call $p\in \R^n$ a \textit{reachable} gradient of $u$ at $x$ if there exists a sequence $\{x_k\}$ such that $u$ is differentiable at $x_k$  for all $k \in \N$ and
$$
\lim_{k\to \infty} x_k=x,\qquad \lim_{k\to \infty}Du(x_k)=p.
$$
The set of all reachable gradients of $u$ at $x$ is denoted by $D^*u(x)$.

\subsection{Lax-Oleinik operator in time-dependent case and \emph{a priori} estimate}



 Let  $f:\R^n\to\R$ be a Lipschitz function. For any $0\leqslant t_1<t_2<+\infty$ and $x_1,x_2\in\R^n$, we define the Lax-Oleinik operator
\begin{equation}\label{T -}
	T^-_{t_1,t_2}f(x_2):=\inf_{z\in \R^n } \{f(z)+A_{t_1,t_2}(z,x_2)  \},
\end{equation}
\begin{equation}\label{T +}
	T^+_{t_1,t_2}f(x_1):=\sup_{y\in \R^n } \{f(y)-A_{t_1,t_2}(x_1,y)  \},
\end{equation}
and denote
\begin{equation}
\begin{split}
	Z(f,t_1,t_2,x_2)=\{z\in\R^n : T^-_{t_1,t_2}f(x_2)= f(z)+A_{t_1,t_2}(z,x_2)\}, \\
	Y(f,t_1,x_1,t_2)=\{y\in\R^n : T^+_{t_1,t_2}f(x_1)= f(y)-A_{t_1,t_2}(x_1,y)\}.
\end{split}
\end{equation}

From Appendix B, we have the following \emph{a priori} estimates:

\begin{Lem}\label{T - T + est1}
	(\rm{proved in Appendix \ref{Appendix B}}) \ Suppose $L$ satisfies \mbox{\rm{(L1)-(L3)}} and $f$ is a Lipschitz function on $\R^n$. Then for any fixed $T>0$, there exists a constant $\lambda_1(T,\mbox{\rm{Lip}}[f])>0$ such that for any $0\leqslant t_1<t_2\leqslant T$ and $x_1,x_2\in\R^n$
	\begin{enumerate}[\rm (1)]
		\item $Z(f,t_1,t_2,x_2) \neq \emptyset $, and for any $z_{t_1,t_2,x_2}\in Z(f,t_1,t_2,x_2)$,
		\begin{equation*}
		|z_{t_1,t_2,x_2}-x_2|\leqslant \lambda_1(T,\mbox{\rm{Lip}}[f])(t_2-t_1).
		\end{equation*}
		
		\item $Y(f,t_1,x_1,t_2) \neq \emptyset $, and for any $y_{t_1,t_2,x_1}\in Y(f,t_1,x_1,t_2) $,
		\begin{equation*}
		|y_{t_1,t_2,x_2}-x_1|\leqslant \lambda_1(T,\mbox{\rm{Lip}}[f])(t_2-t_1).
		\end{equation*}
	\end{enumerate}
	where $\lambda_1(T,K)= \theta_{T}^*( K +1 )+ c_{T} + \overline \theta_{T} (0 ) $ for $T>0$ and $K\geqslant 0$.
\end{Lem}

For $0\leqslant t_1<t_2<+\infty$ and $x_1,x_2\in\R^n$, denote
\begin{align*}
	\Gamma^{t_1,t_2}_{\cdot,x_2}=\{\xi \in W^{1,1}([t_1,t_2],\R^n):\xi(t_2)=x_2 \},\\
	\Gamma^{t_1,t_2}_{x_1,\cdot}=\{\xi \in W^{1,1}([t_1,t_2],\R^n):\xi(t_1)=x_1 \}.
\end{align*}

\begin{Lem}\label{lem:p endpoint}
Suppose $L$ satisfies \mbox{\rm{(L1)-(L3)}}, $f$ is a Lipschitz function on $\R^n$ and $0\leqslant t_1<t_2<+\infty$, $x_1,x_2\in\R^n$.
\begin{enumerate}[\rm (1)]
	\item If $\xi\in\Gamma^{t_1,t_2}_{\cdot,x_2}$ is a minimizer for $T_{t_1,t_2}^{-}f(x_2)$, then
	$
		p(t_1)=L_v(t_1,\xi(t_1),\dot{\xi}(t_1))\in D^{-}f(\xi(t_1)).
	$
    \item If $\xi\in\Gamma^{t_1,t_2}_{x_1,\cdot}$ is a maximizer for $T_{t_1,t_2}^{+}f(x_1)$, then
    $
    	p(t_2)=L_v(t_2,\xi(t_2),\dot{\xi}(t_2))\in D^{+}f(\xi(t_2)).
    $
\end{enumerate}
\end{Lem}
From now on, suppose $u_0$ is a Lipschitz function on $\R^n$ and denote
\begin{equation}\label{eq:u(t,x) rep 1}
	u(t,x)=T^-_{0,t}u_0(x) =\inf_{z\in \R^n} \{u_0(z)+A_{0,t}(z,x)  \}, \quad (t,x) \in [0,+\infty)\times \R^n.
\end{equation}
Actually, we also have the following representation:
 \begin{equation}\label{eq:u(t,x)}
 u(t,x)=\inf_{\xi \in \Gamma^{0,t}_{\cdot, x}} \big\{ u_0(\xi(0))+ \int_0^t L(\tau, \xi(\tau) ,\dot \xi(\tau) )  d \tau\big\} . 
 \end{equation} 

\begin{Pro}\label{pro:properties of u}\cite{Cannarsa_Sinestrari_book}
	The following properties hold.
	\begin{enumerate}[\rm (1)]
		\item $u(t,x)$ is a viscosity solution of \eqref{cauchy equation}.
		\item $u(t,x)$ is locally linear semiconcave on $(0,+\infty)\times \R^n$. More precisely, for any $0<T_1<T_2$ and $R>0$, there exists $C(T_1,T_2,R)>0$ such that $u(t,x)$ is linearly semiconcave on $(T_1,T_2)\times B(0,R)$ with semiconcavity constant $C(T_1,T_2,R)$. Moreover, $C(T_1,T_2,R)$ is continuous with respect to $R$.
		\item For any $(t,x)\in (0,\infty)\times \R^n$ and any minimizer $\xi \in \Gamma^{0,t}_{\cdot ,x} $ of \eqref{eq:u(t,x)}, $u$ is differentiable at $(\tau,\xi(\tau))$ for all $\tau\in (0,t)$.
	\end{enumerate} 
\end{Pro}

Moreover, we have the following result
\begin{Pro}\label{prop:D^*}\cite[Thm 6.4.9]{Cannarsa_Sinestrari_book}
	For any $(t,x)\in (0,+\infty)\times \R^n$, $(q,p)\in D^* u(t,x) $ if and only if there exists a minimizer $\gamma\in\Gamma^{0,t}_{\cdot ,x}$ of \eqref{eq:u(t,x)} such that and $p =L_v(t, x, \dot \gamma(t)) $ and $q=-H(t,x,p)$.
\end{Pro}

\begin{Lem}\label{lem:estimation2}
(\rm{proved in Appendix \ref{Appendix B}}) \ For any fixed $T>0$, there exists $F_0(T)\geqslant 0$ such that $u$ is a Lipschitz function on $(0,T]\times\R^n$ and $\mbox{\rm Lip}\,[u]\leqslant F_0(T)$.
\end{Lem}

Due to Lemma \ref{lem:estimation2} and Lemma \ref{T - T + est1}, we obtain the following estimation:
\begin{Cor}\label{T - T + est3}
For any $T>0$, $0\leqslant t_1<t_2\leqslant T$ and $x_1\in\R^n$, we have
\begin{align*}
	Y(u(t_2,\cdot),t_1,x_1,t_2)\neq\emptyset,
\end{align*}
and for any $y_{t_1,t_2,x_1}\in Y(u(t_2,\cdot),t_1,x_1,t_2)$,
\begin{equation}\label{lambda 2 est }
    |y_{t_1,t_2,x_1}-x_1|\leqslant\lambda_2(T)(t_2-t_1).
\end{equation}
where $\lambda_2(T):=\lambda_1(T,F_0(T))$ for $T>0$ with $\lambda_1$ defined in Lemma \ref{T - T + est1} and $F_0$ defined in Lemma \ref{lem:estimation2}.
\end{Cor}


\subsection{Propagation of singularities}
Recall that $x \in \R^n $ is a cut point of $u$ if no backward calibrated curve ending at $x$ can be extended beyond $x$. We denote by $\mbox{\rm{Cut}}\,(u)$ the set of cut points of $u$. It is well known that $\mbox{\rm{Sing}}\,(u) \subset \mbox{\rm{Cut}}\,(u) \subset \overline{\mbox{\rm{Sing}}\,(u)}$. In the following proposition \ref{propagation of singularitie}, we construct a singular arc starting from any cut point of $u(t,x)$.

\begin{Pro}\label{propagation of singularitie}
	Suppose $L$ is a Tonelli Lagrangian satisfying \mbox{\rm{(L1)-(L3)}}, $H$ is the associated Hamiltonian and $u_0:\R^n\to\R$ is a Lipschitz function.Then for any fixed $x \in \R^n $and $0<t_0 <T $, there exist $t_{x,T} \in (0,T) $ which only depends on $x,T$ such that for all $t_0 \leqslant t_1 \leqslant T-t_{x,T}  $ and $ x_1 \in B(x, \lambda_2(T) t_1 ) $, there is a unique maximum point $y_{t_1,t,x_1}$ of $u(t,\cdot )-A_{t_1,t}(x_1,\cdot)$ for $t\in [t_1,t_1+t_{x,T}]$ and the curve
		\begin{equation}
		\mathbf{y}_{t_1,t_1+t_{x,T},x_1}(t):=\begin{cases}
			x_1 \quad & \mbox{\rm{if}} \ \ t=t_1, \\
			y_{t_1, t,x_1} \quad & \mbox{\rm{if}} \ \ t \in ( t_1,t_1+t_{x,T}].
		\end{cases}
	\end{equation}
	satisfies $\mathbf{y}_{t_1,t_1+t_{x,T},x_1}(t)\in B(x, \lambda_2(T)T)$ for any $t\in[t_1,t_1+t_{x,T}]$, where $\lambda_2$ is defined in Corollary \ref{T - T + est3}. Moreover, if $(t_1,x_1) \in \mbox{\rm{Cut}}(u)$, then $(t,\mathbf{y}(t)) \in \mbox{\rm{Sing}}(u)$ for all $t\in[t_1,t_1+t_{x,T}]$.
\end{Pro}
\begin{proof}
	For any fixed $x\in \R^n , T>0$ , by proposition \ref{pro:properties of u}\,(2),  there exists $C(x,T ) >0$ such that it is a semiconcavity constant for $u$ on $ [t_0,t_0+T]\times  B(x,\lambda_2(T )  T ) $.
	 By (3) of Proposition \ref{A.2}, there exists $C_2(x,T)>0$ such that it is a uniformly convexity constant for $A_{t_1,t}(x_1,\cdot)$ on $ [t_0,t_0+T]\times  B(x,\lambda_2(T)  T ) $. 
	
	Therefore, $u(t,\cdot )-A_{t_1,t}(x_1,\cdot) $ is strictly concave on $\overline B (x,\lambda_2(T)  t) $  for all  $ t \in [t_1,t_1+t_{x,T}] $ provided that we further restrict $t_{x,T}$ in order to have
	$$
	t_{x,T}:= \frac{C_2(x,T) }{2 C(x,T)} .
	$$ 
	We now proof that $y_{t_1, t, x_1} $ is a singular point of $u$ for every $t\in (t_1,t_1+t_{x,T}] $. Let $\xi_{t_1,t,x_1 }\in \Gamma^{t_1,t}_{x_1,y_{t_1, t, x_1} } $ be the unique minimizer for $A_{t_1, t}(x_1,y_{t_1,t,x_1})$ and let
	$$
	p_{t_1,t,x_1}(s):=L_v(s, \xi_{t_1,t,x_1 }(s),\dot \xi_{t_1,t,x_1 }(s) ),\quad s\in [t_1,t_1+t_{x,T}],
	$$ 
	be the associated dual arc. We claim that
	\begin{equation}\label{D+ D* }
		p_{t_1,t,x_1}(t)\in D^+u(t,y_{t_1,t,x_1} ) \backslash D^* u(t ,y_{t_1,t,x_1} )
	\end{equation}

	which in turn yields $y_{t_1,t,x_1}\in \mbox{\rm{Sing}}(u) $. Indeed, if $p_{t_1,t,x_1}(t)\in D^* u(t,y_{t_1,t,x_1} ) $, then by Proposition \ref{prop:D^*}, there would exist a $C^2$ curve $\gamma_{t_1,t,x_1}:(-\infty,t]\rightarrow \R^n $ solving the minimum problem
	\begin{equation}\label{eq:min}
		\min_{\gamma\in W^{1,1}([\tau,t];\R^n )} \left \{ \int_\tau^t L(s,\gamma(s), \dot \gamma(s) ) \ ds +u(\gamma (\tau) ) : \gamma(t)=y_{t_1,t,x_1}   \right \}
	\end{equation}
	for all $\tau \leqslant t $. It is easily to checked that $\gamma_{t_1,t,x_1} $ and $\xi_{t_1,t,x_1} $ coincide on $[t_1,t]$ since both of them are extremal curves for $L$ and satisfy the same endpoint condition at $\gamma_{t_1, t ,x}$ i.e.
	$$
	L_v(t,\xi_{t_1,t,x_1}(t), \dot \xi_{t_1,t,x_1}(t) )= p_{t_1,t,x_1}(t)=L_v(t,\gamma_{t_1, t ,x}(t), \dot \gamma_{t_1, t ,x}(t) ).
	$$
	This leads to a contradiction since $(t_1, x_1 )\in \mbox{\rm{Cut}} (u) $ while $u$ should be smooth at $(t_1, \gamma_{t_1,t,x_1}(t_1) )$ and $\gamma_{t_1, t ,x}(\tau)$ is a backward calibrated curve for $\tau \in [t_1,t] $.
	Thus, \eqref{D+ D* } holds true and $ (t,\mathbf{y}_{t_1,t,x_1}(t)) \in \mbox{\rm{Sing}}(u) $ for all $t\in (t_1,t_1+t_{x,T}] $.
\end{proof}
By Proposition \ref{propagation of singularitie}, for any $ t>0$ and $x'\in \R^n $ , there exist a $t'>t$ which depends on $x'$ such that $\arg \sup_{y\in \R^n } \{ u(s, y)-A_{t,s}(x',y)   \} $ is singleton for any $t< s \leqslant t' $. We can denote that 
 $$
 Y(t, t',x'):=Y\big(u(t',\cdot ) ,t, t',x'\big) =\arg \sup_{y\in \R^n } \{ u(t', y)-A_{t,t'}(x',y)   \}=y_{t, t',x' }.
 $$
By \eqref{lambda 2 est }, it implies that
\begin{equation}\label{Phi Lip}
	|Y(t, t',x')-x'|\leqslant \lambda_2 (T)  (t'-t) \quad \forall \ t_0<t \leqslant t'\leqslant T .
\end{equation}
 
\begin{The}\label{propagation of singularitie global}
	 Suppose $L$ is a Tonelli Lagrangian satisfying \mbox{\rm{(L1)-(L3)}}, $H$ is the associated Hamiltonian and $u_0:\R^n\to\R$ is a Lipschitz function. Then for any fixed $(t_0,x)\in\mbox{\rm{Cut}}(u)$ and $T>t_0$, there exists a curve $\mathbf{x}:[t_0,+\infty)\to\R^n$ with $\mathbf{x}(t_0)=x$, such that $(s,\mathbf{x}(s))\in\mbox{\rm{Sing}}(u)$ for all $s\in [t_0,+\infty)$.
\end{The}

\begin{proof}
For any fixed $(t_0, x) \in \mbox{\rm{Sing}}(u) $ and $T>0$, we denote that 
$$
 \{\Omega_n\}_{n\in \N}:= B(x, \lambda_2(nT) \  n T ) ,
$$  
with $\overline \Omega_n \subset \Omega_{n+1} $ and $\overline \Omega_n $ is compact for all $n$, in addition, $\R^n= \bigcup_n \overline \Omega_n $.

\medskip
\noindent $\mathbf{Step \ I}$ :Uniform Lipschitz estimation of connections of $ Y $.

For any $s>0$, there are a sequence of points $\{ x_j\}_{1\leqslant j \leqslant n}$ and time $\{ s_j \}_{1\leqslant j \leqslant n} $ with $x_j=Y(s_{j-1},s_j,x_{j-1}) $  for any $1\leqslant j \leqslant n $ and $x_0=x$, for $t_0\leqslant s_1\leqslant s_2 \leqslant \dots \leqslant s_n \leqslant s \leqslant  \big\lceil  \frac{s}{T} \big\rceil   T $.

By \eqref{Phi Lip} , we have that
\begin{align*}
	| Y(s_n , s, x_n)-x |\leqslant &\, |Y(s_n,s,x_n)-x_n|+ \sum_{j=1}^n |Y(s_{j-1},s_j,x_{j-1} ) |\\
	\leqslant &\, \lambda_2 \left(\big\lceil  \frac{s}{T} \big\rceil    T  \right) \left(s-s_n+\sum^n_{j=1} (s_j-s_{j-1})    \right) \\
	=&\, \lambda_2 \left(  \big\lceil  \frac{s}{T} \big\rceil   T \right)(s-t_0),
\end{align*}
which means that for any $1\leqslant j \leqslant n $, 
$$
Y(s_j , s, x_j) \in \Omega_{\lceil  \frac{s}{T}\rceil }, \quad   s \in [s_{j},s_{j+1}], 
$$ 
i.e. $| Y(s_j , s, x_j)-x |\leqslant \lambda_2 \left( \big\lceil  \frac{s}{T} \big\rceil    T \right)(s-t_0)  $ for any $1\leqslant j \leqslant n $.

\medskip
\noindent $\mathbf{Step \ II}$: Construction of curve $\mathbf{x} $.  

For $x_1:=x\in \mbox{\rm{Cut}}(u) \cap \Omega_1 $ without loss of generality, then there exists $t_1:=t_{x,T}>0$ such that $Y(t_0,\cdot,x) $ is defined on $[t_0,t_0+t_1]$ by Proposition \ref{propagation of singularitie}. One can extend $Y$ by induction.

For $ Y(t_0,t_0+t_1,x)\in \Omega_1 $, then we define $Y(t_0,s, x)=Y(t_0+t_1,s,Y(t_0,t_0+t_1,x) )$ for all $s\in [t_0+t_1,t_0+2 t_1]  $ ; inductively, if $Y(t_0,\cdot,x)$ is defined on $[t_0,t_0+k t_1]$ such that $Y(t_0,t_0+kt_1,x)\in \Omega_1 $ , then we define that $Y(t_0,s, x)=Y(t_0+kt_1,s, Y(t_0,t_0+kt_1,x) )$ for all $s\in [t_0+kt_1,t_0+(k+1)t_1 ] $. Now, let
\begin{equation}\label{def k_1}
	k_1=\Big\lfloor \frac{T-t_0}{t_1} \Big\rfloor ,
\end{equation}
then $t_0+ k_1 t_1 \leqslant T $ which implies $Y(t_0,s,x)\in \Omega_1$ for any $s \in [t_0,t_0+k_1t_1 ] $ by Step I .

In a similar way, $x_2:=Y(t_0,t_0+k_1t_1 ,x )\in \Omega_1 \cap \mbox{\rm{Sing}}(u) \subset \Omega_2 \cap \mbox{\rm{Sing}}(u) $, we define   
$$
k_2=\Big\lfloor \frac{2T-k_1 t_1-t_0 }{t_2} \Big\rfloor ,
$$
where  $t_2:=t_{x,2T} \leqslant t_1 $ is determined by applying Proposition \ref{propagation of singularitie} to $\Omega_2$.We also conclude that $t_0+k_1t_1+k_2t_2 \leqslant 2T $ which implies that $Y(t_0,s,x)\in \Omega_2 $ for all $s\in [t_0, t_0+k_1t_1+k_2t_2 ] $ by Step I.

Therefore, by induction, for any $i \in \N $, there exists $t_i:=t_{x,iT} \leqslant t_{i-1} $ is determined by applying Proposition \ref{propagation of singularitie} to $\Omega_i$ with $0 < t_i\leqslant t_{i-1} $ , let 
\begin{equation}\label{def k_i}
	k_i=\Big\lfloor \frac{iT-\sum_{j=1}^{i-1} k_jt_j-t_0 }{t_i} \Big\rfloor ,
\end{equation}

We also conclude that $t_0+\sum_{j=1}^{i} k_jt_j \leqslant iT $ which implies that $Y(t_0,s,x)\in \Omega_i $ for all $s \in [t_0, t_0+\sum_{j=1}^{i} k_jt_j  ] $ by Step I.

 Denote that
 $
 x_i:=Y(t_0,t_0+\sum_{j=0}^{i-1} k_j t_j ,x) \in \Omega_{i-1} \cap \mbox{\rm{Sing}}(u) \subset \Omega_{i} \cap \mbox{\rm{Sing}}(u)  . 
 $ 
 This makes us to define an arc $\mathbf{x}:[0, \overline t )\rightarrow \R^n  $ by 
 $$
 \mathbf{x}(s):=Y(t_0,s,x)=Y(t_0+\sum_{j=0}^{i-1} k_j t_j,s,x_i  ) \quad \forall s\in \Big[ t_0+\sum_{j=0}^{i-1} k_j t_j, t_0+\sum_{j=0}^{i} k_j t_j \Big]
 $$
where $\overline t := t_0+\sum^\infty_{j=0} k_j t_j $. It is clear that $\mathbf{x} $ is a generalized characteristic defined on $[0,\overline t)$ and $\mathbf{x}(s) \in \mbox{\rm{Sing}}(u) $ for all $s\in [0,\overline t ) $, by Proposition \ref{propagation of singularitie}.

\medskip  
\noindent  $\mathbf{Step\ III} $: Estimation of time $\overline t $.
  
  To finish the proof, we only need to show that $\overline t =\infty $. Indeed, since $ T> t_1 \geqslant t_2 \geqslant t_3 \geqslant \dots $, we have that 
  \begin{align*}
  	\overline t > &\,t_0+ \sum^n_{j=1} k_j t_j  =t_0+ \sum^{n-1}_{j=1}k_j t_j+  \Big\lfloor \frac{nT-\sum_{j=1}^{n-1} k_j t_j -t_0}{t_n} \Big\rfloor t_n \\
  	\geqslant &\, t_0+ \sum^{n-1}_{j=1}k_j t_j+  \left( \frac{nT-\sum_{j=1}^{n-1} k_j t_j-t_0 }{t_n} -1 \right) t_n   \\
  	=&\, nT-t_n \geqslant nT-t_1   \rightarrow \infty \quad  \text{as} \ n \rightarrow \infty.
  \end{align*} 
  Therefore $ \overline t= \infty $.
\end{proof}

\begin{Rem}
	Example 5.6.7 of \cite{Cannarsa_Sinestrari_book} showed that there exists a counterexample for global propagation of singularities   without condition (L2).
\end{Rem}

\subsection{Local Lipschitz of  singular curve }
 \begin{enumerate}
	\item [\rm{\textbf{(A)}}] For any given $T\in \R^+$, there exists $K(T) \in \R^+$  such that viscosity solution $u(t ,x)$ of \eqref{eq:HJ} is differentiable on $(t,x)\in [0,T]\times \R^n $ and
	$$
	|D_t u(t,x)-D_s u(s,y) |\leqslant K(T)\Big(|t-s|+|x-y|\Big),\quad \forall x,y\in \R^n, s,t \in [0,T].
	$$
\end{enumerate} 

\begin{The}
	If  condition \rm{\textbf{(A)}}  holds, then for any $T>0$ and  $\mathbf{x}(s)$ (see Theorem \ref{propagation of singularitie global}) is a Lipschitz curve on $s\in [0,T]$.
\end{The}

The proof of the theorem above is a direct consequence of following Lemma.

\begin{Lem}\label{Lipschitz lem0}
	For any $T>0$ and $(t_0,x)\in[0,T]\times\R^n$, let $t_{x,T}$ and $\mathbf{y}:[t_0,t_0+t_{x,T}]\to \R^n$ be given by Proposition \ref{propagation of singularitie}. If condition (A) holds, then $\mathbf{y}$ is Lipschitz on $[t_0,t_0+t_{x,T}]$.
\end{Lem}
\begin{proof}
Without loss of generality, we assume that $t_0=0$. Let $\xi_{t}:=\xi_{0,t,x}\in \Gamma_{x,\mathbf{y}(t)}^{0,t}$, $\xi_{s}:=\xi_{0,s,x}\in \Gamma_{x,\mathbf{y}(s)}^{0,s }$ and $\eta:=\xi_{0,t,x}\in \Gamma_{x,\mathbf{y}(s)}^{0,t }$ be minimizers for $A_{0,t}(x,\mathbf{y}(t)),A_{0,s}(x,\mathbf{y}(s)) $ and $A_{0,t}(x,\mathbf{y}(s))$ respectively. Setting $p_t=L_v(t ,\xi_t(t),\dot \xi_t(t)),p_s=L_v(s ,\xi_s(s),\dot \xi_s(s))$ and $p=L_v(t ,\eta(t),\dot \eta(t))$,
we have $(q_s,p_s)\in D^+u(s,\mathbf{y}(s))$ and $(q_t,p_t)\in D^+u(t,\mathbf{y}(t))$.
Hence, by Proposition \ref{A.2}, there exists $C_3(x,T)>0$
such that \begin{align*}
	&\,\frac{C_2(x,T)}{t-s}|\mathbf{y}(t)-\mathbf{y}(s) |^2 \\
	\leqslant &\, \langle p_t-p, \mathbf{y}(t)-\mathbf{y}(s) \rangle =\langle p_t-p_s, \mathbf{y}(t)-\mathbf{y}(s) \rangle +\langle p_s-p, \mathbf{y}(t)-\mathbf{y}(s) \rangle \\
	\leqslant &\,\frac{C_3(x,T)}{t-s} \cdot |\mathbf{y}(t)-\mathbf{y}(s) |\cdot |t-s| + \langle p_t-p_s, \mathbf{y}(t)-\mathbf{y}(s) \rangle \\
	&\, \quad\quad \quad\quad +  \langle q_t-q_s,  t-s \rangle -\langle q_t-q_s,  t-s \rangle \\
	\leqslant &\,\frac{C_3(x,T)}{t-s} \cdot |\mathbf{y}(t)-\mathbf{y}(s) |\cdot |t-s| + C_1(x,T) (|\mathbf{y}(t)-\mathbf{y}(s) |^2 + |t-s|^2  ) 
	\\
	&\, \quad  \quad \quad \quad   \quad  \quad \quad \quad 
	  -\langle q_t-q_s,  t-s \rangle 
\end{align*}
	where  $C_2(x,T)>0$ is a uniformly convexity constant in Proposition \ref{A.2} (3). 
	
Actually, by condition \rm{\textbf{(A)}}, $u(t,x)$ is differentiable with respect to $t$, and
 $$|q_t-q_s |=|D_tu(t,x)-D_su(s,x) |\leqslant  K(T) \big(|t-s|+|\mathbf{y}(t)-\mathbf{y}(s) | \big) . $$
Therefore,
\begin{align*}
	\frac{C_2(x,T)}{t-s}|\mathbf{y}(t)-\mathbf{y}(s) |^2 
	\leqslant &\, \frac{C_3(x,T)}{t-s} \cdot |\mathbf{y}(t)-\mathbf{y}(s) |\cdot |t-s| + C_1(x,T) (|\mathbf{y}(t)-\mathbf{y}(s) |^2 + |t-s|^2  ) \\
	& \, \quad \quad\quad\quad + K(T) \cdot |\mathbf{y}(t)-\mathbf{y}(s) |\cdot |t-s|+K(T)|t-s|^2.
\end{align*}
That is,
\begin{align*}
	\Big( \frac{C_2(x,T)}{t-s}-C_1(x,T) \Big) \Big|\frac{\mathbf{y}(t)-\mathbf{y}(s)}{t-s}\Big|^2 -\Big(\frac{C_3(x,T)}{t-s}  + K(T)\Big)\Big|\frac{\mathbf{y}(t)-\mathbf{y}(s)}{t-s}\Big| \leqslant C_1+K(T).
\end{align*}
Let $t-s$ be sufficiently small such that 
$$ t-s\leqslant\frac{C_2(x,T)}{2C_1(x,T)}, $$
then there exists a constant $C_4$ which only depends on $x,T$ such that
$$
\Big|\frac{\mathbf{y}(t)-\mathbf{y}(s)}{t-s}\Big| \leqslant C_4(x,T).
$$
More precisely, we can take $C_4(x,T):= \frac{ C_3(x,T)}{ C_2(x,T)}+  \frac{ K(T)+ \sqrt{C_1(x,T)+K(T)}}{2C_1(x,T) } $.
\end{proof}

\medskip

\section{Global propagation of singularities for discounted Hamiltonian}\label{Sect.3}

\subsection{Global propagation of singularities for discounted Hamiltonian}
For $\lambda>0$, we consider the Hamilton-Jacobi equation with discounted factor
\begin{equation}\label{eq:discount}\tag{HJ$_\lambda$}
	\lambda v(x)+H( x, Dv(x) )=0 , \quad x\in \R^n.
\end{equation}
where $H $ is a Tonelli Hamiltonian.


\begin{Lem}\label{equ1}\cite[Proposition 3.3]{Chen_Cheng_Zhang2018}
 $v(x) $ is a viscosity solution of \eqref{eq:discount} if and only if 
 $u(t,x)=e^{\lambda t}v(x)$ is a viscosity solution of the Hamilton-Jacobi equation
\begin{equation}\label{eq:dis}
	\begin{cases}
	D_t u+\widehat H (t,x,D_x u)=0,& (t,x)\in(0,+\infty)\times \R^n \\
	u(0,x)=v(x),& x\in \R^n ,
	\end{cases}
\end{equation}
where $\widehat H(t,x,p)=e^{\lambda t}H(x,e^{-\lambda t}p)$. Moreover, for any $(t,x)\in(0+\infty)\times\R^n$, we have  
\begin{equation}\label{eq:sing equ}
	x\in \mbox{\rm{Sing}}(v) \Leftrightarrow  (t,x)\in \mbox{\rm{Sing}}(u), \quad x\in \mbox{\rm{Cut}}(v) \Leftrightarrow  (t,x)\in \mbox{\rm{Cut}}(u). 
\end{equation}
\end{Lem}
\begin{Rem}
	Since $\gamma:[a,b]\to\R^n$ is a calibrated curve   of equation \eqref{eq:discount} with Hamiltonian $H$ if and only if $\gamma:[a,b]\to\R^n$ is a calibrated curve   of  equation \eqref{eq:dis} with Hamiltonian $\widehat H$, then by the definition of $\mbox{\rm{Cut}}(u) $ in Definition \ref{defn:Aubry cut}, $ x\in \mbox{\rm{Cut}}(v) $ and $ (t,x)\in \mbox{\rm{Cut}}(u)$ are equivalent.
\end{Rem}

\begin{The}\label{propagation of singularitie global1}
	Let $H$ be a Tonelli Hamiltonian and $\lambda>0$. Suppose $v:\R^n\to \R$ is the Lipschitz continuous viscosity solution of \eqref{eq:discount}. Then for any fixed $x\in \mbox{\rm{Cut}}(v)$, there exists a locally Lipschitz curve $\mathbf{x}:[0,+\infty)\rightarrow \R^n $ with $\mathbf{x}(0)=x $, such that $ \mathbf{x}(\tau ) \in \mbox{\rm{Sing}}(v) $ for all $\tau \in [0,+\infty) $.
\end{The}
\begin{proof}
	 By using the variable transformation $s=\tau+1\in [1,+\infty)$, we only need to find a locally Lipschitz curve $\mathbf{x}:[1,+\infty)\rightarrow \R^n $ with $\mathbf{x}(1)=x $, such that $ \mathbf{x}(s) \in \mbox{\rm{Sing}}(v) $ for all $s\in [1,+\infty) $. Due to \eqref{eq:sing equ}, $x\in\mbox{\rm{Cut}}(v)$ implies $(1,x)\in\mbox{\rm{Cut}}(u)$. It is easy to check that $\widehat L(t,x,v)=e^{\lambda t}L(x,v) $ satisfies \mbox{\rm(L1)-(L3)}. Therefore, by Theorem \ref{propagation of singularitie global}, there exists a  curve $\mathbf{x}:[1,+\infty) \rightarrow  \R^n $ with $\mathbf{x}(1)=x $, such that $ (s, \mathbf{x}(s)) \in \mbox{\rm{Sing}}(u) $ for all $s \in (1,+\infty) $, that is, $  \mathbf{x}(\tau) \in \mbox{\rm{Sing}}(v) $ for all $\tau \in (0,+\infty) $.
	
	It remains to show that the curve $\mathbf{x}(\tau):[0,+\infty) \rightarrow \R^n $ is locally Lipschitz. Notice that $u(t,x)=e^{\lambda t}v(x)$ is differentiable with respect to $t$ and
\begin{align*}
	D_t u(t,x)= \lambda e^{\lambda t} v(x)= \lambda u(t,x),\qquad (t,x)\in(0,+\infty)\times\R^n.
\end{align*}
For any $(t,x)\in(0,+\infty)\times\R^n$ and $(s,y)\in(0,+\infty)\times\R^n$, by Lemma \ref{lem:estimation2}, we have
\begin{align*}
    |D_t u(t,x)-D_s u(s,y)|=|\lambda u(t,x)-\lambda u(s,y)|\leqslant F_0(T)(|t-s|+|x-y|),
\end{align*}
which implies condition \textbf{(A)} holds. Hence $\mathbf{x}:[0,+\infty) \rightarrow  \R^n $ is locally Lipschitz by Lemma \ref{Lipschitz lem0}.
\end{proof}


\subsection{Homotopy equivalence}
Now, suppose $u:\R^n\to\R$ is the Lipschitz viscosity solution of
\begin{equation}\label{eq:HJs}\tag{HJ$_\lambda$}
	\lambda u(x)+H(x,du(x))=0,\qquad x\in\R^n,
\end{equation}
where $H$ is a Tonelli Hamiltonian and $\lambda>0$. 

\begin{defn}\label{defn:Aubry cut}
	 (Aubry set): We define $\mathcal{I}(u)${\footnote{If $M$ is compact, $\mathcal{I}(u)$ is not empty and can be characterized by conjugate pairs for contact Hamiltonian systems with increasing condition in \cite{Wang_Wang_Yan2019_2}. For noncompact case, the question is still open if $\mathcal{I}(u)\neq \emptyset$ . }}, the Aubry set of $u$, as
	\begin{align*}
		\mathcal{I}(u)=\{x\in\R^n:\text{ there exists a calibrated curve } \gamma:(-\infty,+\infty)\to \R^n \text{ with } \gamma(0)=x\}
	\end{align*}
\end{defn}

In general we have the following inclusions:
$$
\mbox{\rm{Sing}}\,(u) \subset \mbox{\rm{Cut}}\,(u) \subset \R^n\setminus\mathcal{I}(u), \quad \mbox{\rm{Sing}}\,(u) \subset \mbox{\rm{Cut}}\,(u) \subset \overline{\mbox{\rm{Sing}}\,(u)} .
$$

\begin{The}\label{thm:homotopy equivalence}
The inclusions 
		$$
		\mbox{\rm{Sing}}\,(u) \subset \mbox{\rm{Cut}}\,(u) \subset \Big( \R^n\setminus\mathcal{I}(u) \Big) \cap  \overline{\mbox{\rm{Sing}}\,(u)} \subset \R^n\setminus\mathcal{I}(u)
		$$
		are all homotopy equivalences. 
\end{The}

This theorem obviously implies the following corollary (see, for instance, \cite{Dugundji_book})
\begin{Cor}
	For every connected component $C$ of $\R^n\setminus\mathcal{I}(u)$, these three intersections $\mbox{\rm{Sing}}\,(u)\cap C $, $\mbox{\rm{Cut}}\,(u) \cap C $ and $\overline{\mbox{\rm{Sing}}\,(u)} \cap C $ are path connected.
\end{Cor}

\begin{The}\label{thm:locally contractible}\cite[Thm. 1.3]{Cannarsa_Cheng_Fathi2017}
	The spaces $\mbox{\rm{Sing}}\,(u)$ and  $\mbox{\rm{Cut}}\,(u)$ are locally contractible.
\end{The} 

The proof of Theorem \ref{thm:homotopy equivalence} and \ref{thm:locally contractible} needs the following Lemma.

\begin{Lem}\label{lem:homotopy on [0,t]}
	There exists a continuous homotopy $F:\R^n \times [0,+\infty) \to \R^n $ with the following properties:
	\begin{enumerate}[\rm (a)]
		\item for all $x\in\R^n$, we have $F(x,0)=x$;
		\item if $F(x,s)\notin \mbox{\rm{Sing}}(u)$ for some $s>0$ and $x\in\R^n$, then the curve $\sigma\mapsto F(x,\sigma)$ is calibrated on $[0,s]$;
		\item if there exists a calibrated curve $\gamma:[0,s]\to \R^n $ with $\gamma(0)=x$, then $ \sigma \mapsto F(x,\sigma)=\gamma(\sigma) $, for every $\sigma\in[0,s]$.
	\end{enumerate}
\end{Lem}

The proof of Lemma \ref{lem:homotopy on [0,t]} is in Appendix \ref{Appendix C}.
These properties imply:
\begin{Lem}\label{lem:extended F}
\begin{enumerate}[\rm (1)]
	\item $F(\mbox{\rm Cut}\,(u)\times(0,+\infty))\subset\mbox{\rm{Sing}}(u)$;
	\item if $F(x,s)\notin\mbox{\rm{Sing}}\,(u)$ for all $s\in[0,+\infty)$, then $x\in\mathcal{I}(u)$ and $s\mapsto F(x,s)$, $s\in[0,+\infty)$ is a forward calibrated curve with $F(x,0)=x$;
	\item if $x\notin \mathcal{I}(u)$, then $F(x,s)\notin\mathcal{I}(u)$ for every $s\in[0,+\infty)$.
\end{enumerate}
\end{Lem}

Now, for $x\in\R^n$, we define $\tau(x)$ to be the supremum of the $t\geqslant 0$ such that there exists a calibrated curve $\gamma:[0,t]\to\R^n$ with $\gamma(0)=x$.

\begin{Lem}\label{lem:cut time function}
\begin{enumerate}[\rm (i)]
	\item $\tau(x)=0$ if and only if $x\in\mbox{\rm Cut}\,(u)$;
	\item $\tau(x)=+\infty$ if and only if $x\in\mathcal{I}(u)$;
	\item the function $\tau$ is upper semi-continuous.
\end{enumerate}
\end{Lem}

\begin{proof}
(i) and (ii) follows directly from the definition of $\mbox{\rm Cut}\,(u)$ and $\mathcal{I}(u)$. It remains to prove (iii). Indeed, we only need to prove that for any $\tau'>0$ the set $\{x\in\R^n:\tau(x)\geqslant\tau'\}$ is closed. Take any sequence $x_i$ such that $\tau(x_i)\geqslant\tau'$ and $x_i\to x_0$, and let $\gamma_i:[0,\tau']\to\R^n$, $\gamma_i(0)=x_i$ be the associated calibrated curves. By taking a subsequence, we can assume that
\begin{align*}
    \lim_{i\to\infty}Du(x_i)=p_0\in D^{*}u(x_0).
\end{align*}
Notice that $\gamma_i:[0,\tau']\to\R^n$ is the solution of \eqref{eq:H} with initial condition $\gamma_i(0)=x_i$, $p_i(0)=Du(x_i)$. Let $\gamma_0:[0,\tau']\to\R^n$ be the solution of \eqref{eq:H} with initial condition $\gamma_0(0)=x_0$, $p_0(0)=p_0$. It follows that $\gamma_i$ converges to $\gamma_0$ in $C^2$ topology. Thus, we have
\begin{align*}
	e^{\lambda \tau'}u(\gamma_0(\tau'))&=\lim_{i\to\infty}e^{\lambda \tau'}u(\gamma_i(\tau'))=	\lim_{i\to\infty} u(\gamma_i(0)) +\int_{0}^{\tau'} e^{\lambda t}L(\gamma_i(t),\dot{\gamma}_i(t))\ dt\\
	&=u(\gamma_0(0)) +\int_{0}^{\tau'} e^{\lambda t}L(\gamma_0(t),\dot{\gamma}_0(t))\ dt.\\
\end{align*}
This implies $\gamma_0:[0,\tau']\to\R^n$ is a calibrated curve and $\tau(x_0)\geqslant \tau'$. Therefore, the set $\{x\in\R^n:\tau(x)\geqslant\tau'\}$ is closed and the function $\tau$ is upper semi-continuous.
\end{proof}



\begin{proof}[Proof of Theorem \ref{thm:homotopy equivalence}]
	By Lemma \ref{lem:cut time function}, the function $\tau$ is upper semi-continuous and finite on $\R^n\setminus\mathcal{I}(u)$. Thus, by Proposition 7.20 in \cite{Brown:1995aa}, we can find a continuous function $\alpha:\R^n\setminus\mathcal{I}(u)\to(0,+\infty)$ with $\alpha>\tau$ on $\R^n\setminus\mathcal{I}(u)$. We now define $G: (\R^n\setminus\mathcal{I}(u))\times[0,1]\to\R^n\setminus\mathcal{I}(u)$ by 
	$$
	G(x,s)=F(x,s \alpha (x) ).
	$$
	Due to Lemma \ref{lem:homotopy on [0,t]}, Lemma \ref{lem:extended F} and the continuity of $\alpha$, the map $G(x,s)$ is a homotopy of $\R^n\setminus\mathcal{I}(u)$ into itself, such that $G(\R^n\setminus\mathcal{I}(u),1)\subset\mbox{\rm Sing}\,(u)$ and $G(\mbox{\rm Cut}\,(u),(0,1])\subset\mbox{\rm Sing}\,(u)$. Therefore, the time one map of $G$ gives a homotopy inverse for each one of the inclusions
	\begin{align*}
		\mbox{\rm Sing}\,(u)\subset\mbox{\rm Cut}\,(u)\subset\overline{\mbox{\rm Sing}\,(u)}\cap(\R^n\setminus\mathcal{I}(u))\subset\R^n\setminus\mathcal{I}(u).
	\end{align*}
\end{proof}

\subsection{genuine propagation of singularities}
To study genuine propagation of singularities, we have to check that the singular arc $\mathbf{x}$ in Theorem \ref{propagation of singularitie global1} is not a fixed point. As we show below, the following condition about strong  critical point (see, for instance \cite{Cannarsa_Yu2009},\cite{Cannarsa_Cheng3}) can be useful for this purpose. 
\begin{defn}
	We say that $x\in \R^n$ is a strong critical point of a viscosity solution $v$ of \eqref{eq:discount} if 
$$
0 \in \lambda v(x)+H_p(x,D^+v(x)). 
$$
\end{defn}
\begin{Cor}
Let $\mathbf{x}:[0,+\infty ) \to \R^n$ be the singular curve in Theorem \ref{propagation of singularitie global1}. If $x$ is not a strong critical point of $v$, then there exists $t>0$ such that $\mathbf{x}(s)\neq x_0 $ for all $s\in (0,t]$.
\end{Cor}

\medskip

\subsection{Existence of global Lipschitz viscosity solution of \eqref{eq:discount} }
We assume $L=L(x,v):\R^n\times\R^n\to\R$ is a function of class $C^2$ satisfying:
\begin{enumerate}[\rm (L1')]
  \item $L(x,\cdot)$ is strictly convex for all $x\in\R^n$.
  \item There exist $c_1,c_2\geqslant 0$ and two superlinear functions $\theta_1,\theta_2:[0,+\infty)\to[0,+\infty)$ such that
  \begin{align*}
  \theta_2(|v|)+c_2 \geqslant L(x,v) \geqslant \theta_1(|v|)-c_1,\qquad \forall (x,v)\in\R^n\times\R^n.
  \end{align*}
\end{enumerate}

The associated Hamiltonian $H:\R^n\times\R^n\times\R\to\R$ is defined by
\begin{align*}
	H(x,p)=\sup_{v\in\R^n}\{\langle p,v\rangle-L(x,v)\},\qquad (x,p)\in\R^n\times\R^n.
\end{align*}
%



\begin{The}\label{thm:existence of neg wkam solu}
Suppose $L:\R^n\times\R^n\to\R$ satisfies \mbox{\rm (L1')-(L2')} and $\lambda>0$. Then there exists $u_0:\R^n\to\R$ such that $u_0$ is the unique   bounded and Lipschitz viscosity solution of \eqref{eq:HJs} on $\R^n$.
\end{The}


\begin{Rem}
	In Theorem \ref{propagation of singularitie global1}, we suppose $u:\R^n\to \R$ is a globally Lipschitz continuous viscosity solution of equation \eqref{eq:discount}. Actually, the viscosity solutions of equation \eqref{eq:discount} are not always globally Lipschitz continuous. There is a counterexample as follows:
\begin{equation}\label{eq:counterexample global Lipschitz}
u(x)+\frac{1}{2} |Du(x)|^2=0, \quad x\in \R^n,
\end{equation}
Obviously, $u_1(x)=-\frac{1}{2}x^2$ and $u_2(x)\equiv 0$ are both viscosity solutions of \eqref{eq:counterexample global Lipschitz}. But $u_1(x)=-\frac{1}{2}x^2$ is not globally Lipschitz on $\R^n$ and $u_2(x)\equiv 0$ is the unique globally Lipschitz viscosity solution of \ref{eq:counterexample global Lipschitz}.
\end{Rem}

For any function $u:\R^n\to\R$ and $t>0$, we define the Lax-Oleinik operator (See \cite{CCJWY2020})
\begin{equation}\label{eq:Tt-}
	T_{t}^{-}u(x)=\inf_{\xi\in \Gamma_{\cdot ,x}^{0,t} } \Big\{ e^{-\lambda t} u(\xi(0))+\int_{0}^{t}e^{\lambda (s-t)} L(\xi,\dot{\xi})\ ds\Big\}, \qquad x\in\R^n.
\end{equation}
Recall some properties of $T_t^-$ as follows:
\begin{Lem}\label{lem:T properties}
\begin{enumerate}[\rm (1)]
	\item For any $u:\R^n\to\R$ and $t_1,t_2>0$, we have
	\begin{align*}
		T_{t_1}^{-}T_{t_2}^{-}u=T_{t_1+t_2}^{-}u.
	\end{align*}
    \item Set $u_i:\R^n\to\R,\ i=1,2$. If $u_1\leqslant u_2$, then there holds
    \begin{align*}
    	T_{t}^{-}u_1\leqslant T_{t}^{-}u_2,\qquad \forall t>0
    \end{align*}
    \item Suppose $u_i:\R^n\to\R,\ i=1,2$ are bounded on $\R^n$. Then we have
    \begin{align*}
    	\|T_{t}^{-}u_1-T_{t}^{-}u_2\|_{\infty}\leqslant e^{-\lambda t}\|u_1-u_2\|_{\infty},\qquad \forall t>0.
    \end{align*}
    \item For any $t>0$, $u\leqslant T_{t}^{-}u$ if and only if for any absolutely continuous curve $\gamma:[a,b]\to \R^n $, there holds 
   \begin{equation}\label{dominated}
   	e^{\lambda b}u(\gamma(b)) \leqslant 	e^{\lambda a}u(\gamma(a)) +\int_a^b e^{\lambda t}L(\gamma(t),\dot \gamma(t)) \ dt.
   \end{equation}
\end{enumerate}
\end{Lem}

\begin{Lem} \label{lem:T bounded contract}
There exists positive constants $K_1=c_1/\lambda$, $K_2=(\theta_2(0)+c_2)/\lambda$ such that
$$
-K_1\leqslant	T_{t}^{-}(-K_1)(x)\leqslant K_2,\qquad \forall t>0,x\in\R^n.
$$
\end{Lem}

\begin{proof}
	On the one hand, for any $t>0$, $x\in\R^n$, consider the curve $\xi(s)\equiv x$ in \eqref{eq:Tt-}, then  we obtain that
\begin{align*}
	T_t^-(-K_1)(x)\leqslant &\, -e^{-\lambda t}K_1+
	\int_{0}^{t}e^{\lambda(s-t)} L(x,0)\ ds\\
	 \leqslant &\, \int_{0}^{t}e^{(s-t)\lambda}(\theta_{2}(0)+c_2)\ ds \\
	 \leqslant &\, \frac{\theta_2(0)+c_2}{\lambda}=K_2. 
\end{align*}
On the other hand, let $\eta \in \Gamma_{\cdot ,x}^{0,t}$ be a minimizer for \eqref{eq:Tt-}. It follows that
\begin{align*}
	T_t^-(-K_1)(x)=&\, -e^{-\lambda t} K_1+\int_{0}^{t} e^{\lambda (s-t)} L(\eta,\dot{\eta})\ ds \\
	 \geqslant &\,  -e^{-\lambda t} K_1+ \int_{0}^{t}e^{\lambda (s-t)}\big(\theta_{1}(|\dot{\eta}(s)|)-c_1 \big) \ ds \\
	 \geqslant &\,  -e^{-\lambda t} K_1 - c_1 \int_{0}^{t}e^{\lambda (s-t)}\ ds \\
	 =&\, -e^{-\lambda t} K_1 - (1-e^{-\lambda t} ) K_1=-K_1.
\end{align*}
This completes our proof.
\end{proof}

Similar to \cite[lemma 2.2]{MR3135343}, we show that bounded subsolutions is Lipschitz on $\R^n$ as follows:
\begin{Lem}\label{lem:bounded dom by L implies lip}
Suppose $u:\R^n\to\R$ is bounded on $\R^n$ and $u\leqslant T_{t}^{-}u$ for any $t>0$. Then $u$ is Lipschitz on $\R^n$.
\end{Lem}

\begin{proof}
For $x,y\in\R^n$, $x\neq y$, consider the curve $\xi(s)=x+s\frac{(y-x)}{|y-x|},\ s\in[0,|y-x|]$. Then $u\leqslant T_{t}^{-}u$ implies that
\begin{equation}\label{pf:2.5 1}
    u(y)\leqslant e^{-\lambda |y-x|} u(x)+\int_{0}^{|y-x|}e^{\lambda (s- |y-x|)} L \Big(x+s\frac{(y-x)}{|y-x|},\frac{(y-x)}{|y-x|} \Big) \ ds,
\end{equation}
It follows that
\begin{align*}
	&\,u(y)-u(x)\\
	\leqslant &\, \frac{1-e^{-\lambda |y-x|}}{\lambda}\cdot ( -\lambda u(x))+\int_{0}^{|y-x|}e^{\lambda (s- |y-x|)} L \Big(x+s\frac{(y-x)}{|y-x|},\frac{(y-x)}{|y-x|} \Big) \ ds\\
	= &\, \int_{0}^{|y-x|}e^{\lambda (s- |y-x|)}\Bigg[ L \Big(x+s\frac{(y-x)}{|y-x|},\frac{(y-x)}{|y-x|} \Big) - \lambda u(x) \Bigg] \ ds \\
	\leqslant &\,  \int_{0}^{|y-x|}e^{\lambda(s-|y-x|)}\Big(\theta_{2}(1)+c_2+\lambda \|u\|_{\infty}\Big) \ ds \\
	\leqslant &\, (\theta_{2}(1)+c_2+\lambda \|u\|_{\infty})|y-x|.
\end{align*}
Similarly, there holds $u(x)-u(y)\leqslant (\theta_{2}(1)+c_2+\lambda \|u\|_{\infty})|y-x|$. Therefore,
\begin{align*}
	|u(y)-u(x)|\leqslant(\theta_{2}(1)+c_2+\lambda \|u\|_{\infty})|y-x|,\qquad \forall x,y\in\R^n.
\end{align*}
\end{proof}

Following Fathi (\cite{Fathi1997_1,Fathi1997_2}), $u\in C^0(\R^n,\R)$ is called a weak-KAM solution of \eqref{eq:HJs} if
	\begin{align*}
		T_{t}^{-}u=u,\qquad \forall t>0.
	\end{align*}

\begin{Lem}\label{lem:neg wkam solu is viscosity solu}
Suppose $u:\R^n\to\R$ is bounded on $\R^n$. Then $u$ is a  weak-KAM solution of \eqref{eq:HJs} if and only if $u$ is a viscosity solution of \eqref{eq:HJs}.
\end{Lem}

\medskip
\begin{proof}[Proof of Theorem \ref{thm:existence of neg wkam solu}]
Due to Lemma \ref{lem:T bounded contract}, we have that
\begin{align*}
	-K_1\leqslant	T_{t}^{-}(-K_1)(x)\leqslant K_2,\qquad \forall t>0,x\in\R^n,
\end{align*}
and $T_{t}^{-}(-K_1)$ is non-decreasing for $t>0$ by (1) and (4) of Lemma \ref{lem:T properties}.
It follows that $\|T_{1}^{-}(-K_1)-(-K_1)\|_{\infty}\leqslant K_1+K_2$. Using Lemma \ref{lem:T properties} (3), we obtain that for any $k\in \N$,
\begin{align*}
	\|T_{k+1}^{-}(-K_1)-T_{k}^{-}(-K_1)\|_{\infty} \leqslant &\, e^{-k\lambda}\|T_{1}^{-}(-K_1)-(-K_1)\|_{\infty} \\
	\leqslant &\, e^{-k\lambda}(K_1+K_2),
\end{align*}
which implies
\begin{align*}
\sum_{k=1}^{+\infty}\|T_{k+1}^{-}(-K_1)-T_{k}^{-}(-K_1)\|_{\infty}\leqslant\sum_{k=1}^{+\infty}e^{-k\lambda}(K_1+K_2)=\frac{e^{-\lambda}(K_1+K_2)}{1-e^{-\lambda}}<+\infty.
\end{align*}
Therefore, there exists a unique $u_0:\R^n\to\R$ such that
\begin{equation}\label{pf:mian 1}
	\lim_{t\to+\infty}\|T_{t}^{-}(-K_1)-u_0\|_{\infty}=0
\end{equation}
with
\begin{equation}\label{pf:mian 2}
	-K_1\leqslant u_0(x)\leqslant K_2,\qquad \forall x\in\R^n.
\end{equation}
For any $t'>0$, by \eqref{pf:mian 1} and Lemma \ref{lem:T properties} (3) we have
\begin{equation}\label{pf:mian 3}
	T_{t'}^{-}u_0=T_{t'}^{-}\lim_{t\to+\infty}T_{t}^{-}(-K_1)=\lim_{t\to+\infty}T_{t'}^{-}T_{t}^{-}(-K_1)=u_0.
\end{equation}
\eqref{pf:mian 2}, \eqref{pf:mian 3} and Lemma \ref{lem:bounded dom by L implies lip} implies $u_0$ is Lipschitz on $\R^n$. Now we know that $u_0$ is a  weak-KAM solution of \eqref{eq:HJs} which is bounded and Lipschitz on $\R^n$. By Lemma \ref{lem:neg wkam solu is viscosity solu}, $u_0$ is also a viscosity solution of \eqref{eq:HJs}. The uniqueness of $u_0$ is a direct consequence of Lemma \ref{lem:T properties} (3). This completes the proof of Theorem \ref{thm:existence of neg wkam solu}. 
\end{proof}

\appendix
\section{Regularity properties of fundamental solutions}\label{Appendix A}
Here we collect some relevant regularity results with respect to the fundamental solution of \eqref{cauchy equation}. The proofs of these regularity results are similar to those in \cite{Cannarsa_Cheng3} in autonomous case.
\begin{Pro}\label{A.1}
	Suppose L satisfies condition \mbox{\rm(L1)-(L3)}. Then for any $T>0$, $0\leqslant s<t\leqslant T$, $x,y\in\R^n$, and any minimizer $\xi\in\Gamma^{s,t}_{x,y}$ for $A_{s,t}(x,y)$, we have
	\begin{align*}
		\sup_{\tau\in[s,t]}|\dot\xi(\tau)|\leqslant\kappa(T,\frac{|x-y|}{t-s}).
	\end{align*}
	where $\kappa:(0,+\infty)\times (0,+\infty) \rightarrow (0,+\infty) $ is nondecreasing.
\end{Pro}

Now, for $(s,x)\in[0,+\infty)\times\R^n$, $\lambda>0$ and $\tau>0$, let
\begin{align*}
	S_{\lambda}(s,x,\tau)=\{(t,y)\in\R\times\R^n:s<t<s+\tau,|y-x|<\lambda(t-s)\}.
\end{align*}

\begin{Pro}\label{A.2}
	Suppose L satisfies condition \mbox{\rm(L1)-(L3)}.Then for any fixed $T>0$, $R>0$ and $\lambda>0$, there exists $t_0(s,x,T,R,\lambda)>0$ such that for any $(s,x)\in[0,T]\times\bar{B}(0,R)$
	\begin{enumerate}[\rm (1)]
		\item The function $(t,y)\mapsto A_{s,t}(x,y)$ is semiconcave on the cone $S_{\lambda}(s,x,t_0(T,R,\lambda))$ and there exists $C_0(T,R,\lambda)>0$ such that for all $(t,y)\in S_{\lambda}(s,x,t_0(T,R,\lambda))$, $h\in[0,\frac{1}{2}(t-s))$ and $z\in B(0,\lambda(t-s))$ we have that
		\begin{align*}
			A_{s,t+h}(x,y+z)+A_{s,t-h}(x,y-z)-2A_{s,t}(x,y)\leqslant \frac{C_0(T,R,\lambda)}{t-s}(h^2+|z|^2).
		\end{align*}
		\item The function $(t,y)\mapsto A_{s,t}(x,y)$ is semiconvex on the cone $S_{\lambda}(s,x,t_0(T,R,\lambda))$ and there exists $C_1(T,R,\lambda)>0$ such that for all $(t,y)\in S_{\lambda}(s,x,t_0(T,R,\lambda))$, $h\in[0,\frac{1}{2}(t-s))$ and $z\in B(0,\lambda(t-s))$ we have that
		\begin{align*}
			A_{s,t+h}(x,y+z)+A_{s,t-h}(x,y-z)-2A_{s,t}(x,y)\geqslant -\frac{C_1(T,R,\lambda)}{t-s}(h^2+|z|^2).
		\end{align*}
		\item For all $t\in(s,s+\tau]$, the function $A_{s,t}(x,\cdot)$ is uniformly convex on $B(x,\lambda(t-s))$, and there exists $C_2(T,R,\lambda)>0$ such that for all $y\in B(x,\lambda(t-s))$ and $z\in B(0,\lambda(t-s))$ we have that
	    \begin{align*}
	    	A_{s,t}(x,y+z)+A_{s,t}(x,y-z)-2A_{s,t}(x,y)\geqslant \frac{C_2(T,R,\lambda)}{t-s}|z|^2.
	    \end{align*}
        Moreover, $C(T,R,\lambda)$ is continuous with respect to $R$.
	\end{enumerate}
\end{Pro}
\begin{Rem}
	In this paper, for any fixed $T>0$, we choose $\lambda:=\lambda_2(T)$ and $R:=\lambda_2(T)T$, where $\lambda_2(T)$ is defined in Lemma \ref{lem:estimation2}. Assume that $x\in S_\lambda(0,0,s)$, then by Lemma \ref{lem:estimation2} and the definition of $S_\lambda(s,x,\tau)$, we can let $t_0(s,x,T,R)=T-s$ and $C_i(T,R,\lambda)$ only depends on $T$ for $i=1,2,3$.
	
	Moreover, if we consider the domain $[0,T]\times \bar{B}(x_0,R)$ for any $x_0\in\R^n$, then  $C_i(T,R,\lambda)$ only depends on initial point $x_0$ and $T$ for $i=1,2,3$.
\end{Rem}


\begin{Pro}\label{A.5}
	Suppose L satisfies condition \mbox{\rm(L1)-(L3)}. Then for any fixed $T>0$, $R>0$, $\lambda>0$ and $(s,x)\in[0,T]\times \bar{B}(0,R)$, the function $(t,y)\mapsto A_{s,t}(x,y)$ is of class $C^{1,1}_{loc}$ on the cone $S_{\lambda}(s,x,t_0(T,R,\lambda))$, where $t_0(T,R,\lambda)$ is that in Proposition \ref{A.2}. Moreover, we have
	\begin{align*}
		& D_y A_{s,t}(x,y)=L_v(t,\xi(t),\dot \xi(t) ),\\
		& D_x A_{s,t}(x,y)=-L_v(s,\xi(s),\dot \xi(s) ),\\
		& D_t A_{s,t}(x,y)=-E_{s,t,x,y},
	\end{align*}  
	where $\xi\in \Gamma_{x,y}^{s,t} $ is the unique minimizer for $A_{s,t}(x,y)$ and 
    \begin{align*}
    	E_{s,t,x,y}:=H(t,\xi(t),p(t))
    \end{align*}
	is the energy of the Hamilton trajectory $(\xi,p)$ with 
	$$
	p(\tau):=L_v(\tau,\xi(\tau),\dot\xi(\tau)),\quad \tau\in[s,t].
	$$ 
\end{Pro}

\section{Proof of Lemma \ref{T - T + est1} and Lemma \ref{lem:estimation2}}\label{Appendix B}

The convex conjugate of a superlinear function $\theta_T$ is defined as
$$
\theta_T^*(s)=\sup_{r\geqslant 0 } \{rs - \theta_T(r)  \}, \quad   s \geqslant 0 .
$$ 
In view of the superlinear growth of $\theta_T$, it is clear that $\theta_T^* $ is well defined and satisfies 
$$
\theta_T(r)+\theta_T^*(s) \geqslant rs, \quad r,s \geqslant 0,
$$  
which in turn can be used to show that $\theta_T^*(s)/s \rightarrow +\infty $ as $s\rightarrow +\infty $.

\medskip
\begin{proof}[Proof of Lemma \ref{T - T + est1}]
	For item (1), let $k= \mbox{\rm{Lip}}[f]+1 $. Then for any $x\in \R^n $, $0\leqslant t_1<t_2\leqslant T$ and $z \in \R^n $, we have
	\begin{align*}
		A_{t_1,t_2}(z, x)=&\, \inf_{\xi \in \Gamma_{z, x}^{t_1,t_2} } \int_{t_1}^{t_2} L(s,\xi,\dot \xi )\ ds\\
		\geqslant &\, \inf_{\xi \in \Gamma_{z, x}^{t_1,t_2} } \int_{t_1}^{t_2} \theta_T(|\dot \xi|) \ ds - c_T (t_2-t_1)  \\
		\geqslant &\, \inf_{\xi \in \Gamma_{z, x}^{t_1,t_2} } k \int_{t_1}^{t_2} |\dot \xi |  \ ds-(\theta_T^*(k)+c_T )(t_2-t_1) \\
		\geqslant &\, k|z-x|-(\theta_T^*(k)+c_T )(t_2-t_1).
	\end{align*}
	Therefore,
	\begin{align*}
		&f(x)+A_{t_1,t_2}(x,x)-f(z)-A_{t_1,t_2}(z,x) \\
		\leqslant &\, \mbox{\rm Lip}[f] \cdot |z-x|-k|z-x|+( \theta_{T}^*(k)+  c_{T} )\ (t_2-t_1) +  \int_{t_1}^{t_2} L(s,x,0) \ ds \\
		\leqslant &\, -|z-x|+ [ \theta_{T}^*(k)+ c_{T} + \overline \theta_{T} (0 )  ] \ (t_2-t_1) \ .
	\end{align*}	 
	Now, taking $\lambda_1=\theta_{T}^*(k)+ c_{T} + \overline \theta_{T} (0 )$, it follows that
	\begin{equation}\label{Lambda 1}
		\Lambda_{t_1,t_2}^x:=\{z: f(z)+A_{t_1,t_2}(z, x) \leqslant  f(x)+A_{t_1,t_2}(x,x)   \}\subset \overline B\big(x, \lambda_1(t_2-t_1) \big).
	\end{equation}
	Therefore $\Lambda_{t_1,t_2}^x$ is compact and the infimum in \eqref{T -} is attained, i.e., $Z(f,t_1,t_2,x) \neq \emptyset$. Moreover, due to \eqref{Lambda 1}, for any $z_{t_1,t_2,x}\in Z(f,t_1,t_2,x) $, we have
	$$ 
    |z_{t_1,t_2,x}-x|\leqslant\lambda_1(t_2-t_1).
	$$  	
	For item (2), A similar result holds for the sup-convolution defined in \eqref{T +}.
\end{proof}

\medskip
\begin{proof}[Proof of Lemma \ref{lem:estimation2}]
    Set $(t,x)\in(0,T]\times\R^n$ is a differentiable point of $u$. Due to Proposition \ref{pro:fundamental solution} and Proposition \ref{prop:D^*}, the solution of \eqref{eq:H} with terminal condition 
    \begin{equation*}
    	\begin{cases}
    		\xi(t)=x\\
    		p(t)=\nabla u(t,x)
    	\end{cases}
    \end{equation*}
    is the unique minimizer for $u(t,x)$. Lemma \ref{T - T + est1} (1) implies $|\xi(0)-x|\leqslant \lambda_1(T,\mbox{\rm Lip}[u_0])t$. Now, denote that
    \begin{align*}
    	E(s):=H(s,\xi(s),p(s)),\qquad s\in[0,t].
    \end{align*}
    By Lemma \ref{lem:p endpoint}, we know that $p(0)\in D^{-}u_0(\xi(0))$. This implies $|p(0)|\leqslant \mbox{\rm Lip}\,[u_0]$ and
    \begin{align*}
        E(0)=H(0,\xi(0),p(0))\leqslant \theta^*_T(|p(0)|)+c_T \leqslant \theta^*_T(\mbox{\rm Lip}\,(u_0))+c_T.
    \end{align*}
    Notice that
    \begin{align*}
    	&\frac{d}{ds}E(s)=\frac{d}{ds}H(s,\xi(s),p(s))=H_t+H_x\cdot\dot{\xi}(s)+H_p\cdot\dot{p}(s)\\
       =&H_t+H_x\cdot H_p+H_p\cdot(-H_x)=H_t(s,\xi(s),p(s))=-L_t(s,\xi(s),\dot{\xi}(s)).
    \end{align*}
    Thus, we have
    \begin{align*}
    	 E(t)&=E(0)+\int_{0}^{t}\frac{d}{ds}E(s)\ ds\\
             &=E(0)+\int_{0}^{t}L_t(s,\xi(s),\dot\xi(s))\ ds\\
    	&\leqslant E(0)+\int_{0}^{t}\Big(\widetilde C_1(T)+\widetilde C_2(T)L(s,\xi(s),\dot\xi(s))\Big)\ ds\\
    	&\leqslant E(0)+t\ \widetilde C_1(T)+\widetilde C_2(T)\int_{0}^{t} L(s,\xi(0)+s(x-\xi(0)),\frac{x-\xi(0)}{t})\ ds\\
    	&\leqslant E(0)+\widetilde C_1(T)T+\widetilde C_2(T)\int_{0}^{t}\overline\theta_T(|\frac{x-\xi(0)}{t}|)\ ds\\
    	&\leqslant \theta^*_T(\mbox{\rm Lip}\,(u_0))+c_T+\widetilde C_1(T)T+\widetilde C_2(T)T\cdot\overline\theta_T( \lambda_1(T,\mbox{\rm Lip}[u_0])t).
    \end{align*}
    Since 
    \begin{align*}
        |\nabla u(t,x)|\leqslant \overline{\theta}_{T}^{*}(|\nabla u(t,x)|)+\overline{\theta}_{T}(1)\leqslant H(t,x,\nabla u(t,x))+\overline{\theta}_{T}(1)=E(t)+\overline{\theta}_{T}(1),
    \end{align*}
    it follows that
    \begin{align*}
    	|\nabla u(t,x)|\leqslant& \theta^*_T(\mbox{\rm Lip}\,(u_0))+c_T+\widetilde C_1(T)T+\widetilde C_2(T)T\cdot\overline\theta_T( \lambda_1(T,\mbox{\rm Lip}[u_0])t)+\overline{\theta}_{T}(1)\\
    	:=&F_1(T).
    \end{align*}
    By proposition \ref{pro:properties of u} (1), we obtain
    \begin{align*}
    	|u_t(t,x)|&=|-H(t,x,\nabla u(t,x))|\leqslant \theta_{T}^{*}(|\nabla u(t,x)|)+c_0+|\bar{\theta}_{T}^{*}(|\nabla u(t,x)|)|\\
    	&\leqslant \theta_{T}^{*}(F_1(T))+c_0+|\bar{\theta}_{T}^{*}(F_1(T))|:=F_2(T).
    \end{align*}
    Therefore,
    \begin{align*}
    	|Du(t,x)|=(|\nabla u(t,x)|^2+|u_t(t,x)|^2)^{\frac{1}{2}}\leqslant \big(F_1^2(T)+F_2^2(T)\big)^{\frac{1}{2}}:=F_0(T).
    \end{align*}
    Combing this with Proposition \ref{pro:properties of u} (2), we conclude that $u$ is a Lipschitz function on $(0,T]\times\R^n$ and $\mbox{\rm Lip}\,[u]\leqslant F_0(T)$.
\end{proof}

\section{Proof of Lemma \ref{lem:homotopy on [0,t]}}\label{Appendix C}
We define $F:\R^n \times [0,t]\to \R^n$ as:
$$
F(x,s)= \mathbf{x}_x(s), \quad  s \in [0,+\infty),
$$ 
where $\mathbf{x}_x:[0,+\infty)\to \R^n$ is that in Theorem \ref{propagation of singularitie global}.

Similar to \cite[Lemma 2.1]{Cannarsa_Cheng_Fathi2017}, $F(x,s)$ has the properties (a), (b) and (c) stated in Lemma \ref{lem:homotopy on [0,t]}. Due to Theorem \ref{propagation of singularitie global1},  $\mathbf{x}_x(s)$ is a locally Lipschitz curve. Thus, to prove that $F(x,s)=\mathbf{x}_x(s)$ is continuous, it remains to show $\mathbf{x}_x(s)$ is continuous with respect to $x$. We prove it in Lemma \ref{cor continuous}. 

	\begin{Lem}\label{lem:C.1}
		For any fixed $x\in \R^n $ and $t>0$, let $t_{x,T}>0$ be defined in Lemma \ref{propagation of singularitie}. Then for any $0<s<t<T$ with $t-s\leqslant t_{x,T} $, the map $z\mapsto y_{s,t,z}$ is Lipschitz on $B(x,\lambda_2(T)t_{x,T})$ with Lipschitz constant $K_1(x,T)$ which only depends on $x,T$. 
	\end{Lem}
	\begin{proof}
		For any $x_1,x_2$ with $|x_1-x_2|<\lambda_2(T)t_{x,T}$ and $|x_1-x_2|<r$, we denote by $ y_{s,t,x_1}$ and $y_{s,t,x_2}$ the unique maximizers of $T^+_{s,t} u(t,x_1)$ and $T^+_{s,t} u(t,x_2)$ respectively. Notice that  $A_{s,t}(x_1,\cdot )$ is uniformly convex in the ball $B(x_1, 2 \lambda_2(T) t_{x,T} )$ with convexity constant $C_2(x,T)/t$ for $t\in (0,t_T)$.Then we have
		$$
		\frac{C_2(x,T)}{t}|y_{s,t,x_1}-y_{s,t,x_2}|^2 \leqslant A_{s,t}(x_1,y_{s,t,x_1} )-A_{s,t}(x_1,y_{s,t,x_2} ) - D_yA_{s,t}(x_1,y_{s,t,x_2})(y_{s,t,x_1}-y_{s,t,x_2}).
		$$
		On the other hand, since $C(T)$ is the semiconcave constant of $u(t,\cdot )$ for $t\in [0,T]$, we have
		\begin{align*}
			&\,A_{s,t}(x_1,y_{s,t,x_1} )-A_{s,t}(x_1,y_{s,t,x_2} ) \leqslant u(t,y_{s,t,x_1} )- u(t,y_{s,t,x_2} )
			\\
			\leqslant &\, D_yA_{s,t}(x_2,y_{s,t,x_2})(y_{s,t,x_1}-y_{s,t,x_2}) +C(x,T)|y_{s,t,x_1}-y_{s,t,x_2}|^2.
		\end{align*}
		Notice that $D_y A_t(x_2,y_{s,t,x_2})\in\nabla^+ u(t,y_{s,t,x_2})$. Thus,
		\begin{align*}
			&\,\frac{C_2(x,T)}{t-s}|y_{s,t,x_1}-y_{s,t,x_2}|^2 \\
			\leqslant &\, \Big( D_yA_{s,t}(x_2,y_{s,t,x_2})- D_yA_{s,t}(x_1,y_{s,t,x_2}) \Big) (y_{s,t,x_1}-y_{s,t,x_2})  + C(x,T)|y_{s,t,x_1}-y_{s,t,x_2}|^2 \\
			\leqslant &\,  \frac{C_0(x,T)}{t-s} |x_1-x_2|\cdot |y_{s,t,x_1}-y_{s,t,x_2} |+C(x,T)|y_{s,t,x_1}-y_{s,t,x_2}|^2.
		\end{align*}
		Therefore, due to $t_{x,T} := \frac{C_2(x,T)}{2 C(x,T)} $ , for any $0<t-s\leqslant t_{x,T}$, we have
		$$
		|y_{s,t,x_1}-y_{s,t,x_2}|\leqslant \frac{C_0(x,T)}{C_2(x,T)-C(x,T)\cdot (t-s)} |x_1-x_2 |\leqslant \frac{2C_0(x,T)}{C_2(x,T)} |x_1-x_2|:=K_1(x,T) |x_1-x_2|.
		$$
	\end{proof}
	\begin{Lem}\label{cor continuous}
		For any $t>0$, $\mathbf{x}_x(t)$ is continuous with respect to $x$. 
	\end{Lem}
	\begin{proof}
		Let $\Omega_n(x):=\{B(x,\lambda_2(nT)nT) \}$ and assume $t_{x_1,T}\geqslant t_{x_2,T}$. By Theorem \ref{propagation of singularitie global}, we construct an new curve $\widetilde \mathbf{x}_{x_1}(t)$ defined by $\Omega_n(x_1)$ and $\widetilde t_{x_1,T}:=t_{x_2,T}$.   
		Then, for $(n-1)T \leqslant t <nT$ with $n \in\N$, 
		 by Lemma \ref{lem:C.1}, we have
		\begin{equation}\label{eq:C.2 1}
		|\widetilde \mathbf{x}_{x_1}(t)- \mathbf{x}_{x_2}(t)|
		\leqslant K_0(nT) \cdot |x_1-x_2|, \quad \forall |x_1-x_2|<K_1(T)t_{x_2,T}  .
		\end{equation}
		where $K_0(T)= (K_1(T))^{k_1}$ and $k_1$ depends on $\widetilde t_{x_1,T}=t_{x_2,T}$.

		On the other hand, due to Theorem \ref{propagation of singularitie}, Proposition \ref{A.2} and Proposition \ref{pro:properties of u}(2),
		$$
		t_{x,T}= \frac{C_2(x,T)}{2C(x,T)}
		$$
		is continuous with respect to $x$. Therefore, for $0 \leqslant t <T$, we have
		$$
		|\widetilde \mathbf{x}_{x_1}(t)- \mathbf{x}_{x_1}(t)| \leqslant \widetilde K_0(T)\cdot |t_{x_1,T}-t_{x_2,T}|,
		$$
		where $\widetilde K_0(T):=(K_1(T))^{ k_1} $ and $k_1$ depends on $t_{x_1,T}$. Combining this with \eqref{eq:C.2 1}, one obtain that for any $t\in [0,T)$ ,
		\begin{align*}
			|\mathbf{x}_{x_2}(t) -\mathbf{x}_{x_1}(t)|\leqslant &\, | \widetilde \mathbf{x}_{x_1}(t)- \mathbf{x}_{x_2}(t)|+ |\widetilde \mathbf{x}_{x_1}(t)- \mathbf{x}_{x_1}(t) |\\
			\leqslant &\, K_0(T) \cdot |x_1-x_2|+ \widetilde K_0(T)\cdot |t_{x_1,T}-t_{x_2,T}|.
		\end{align*}
		Similarly, for $t\in [(n-1)T,nT)$, there exists constant $K_0(n,T)$ and $\widetilde K_0(n,T)$ such that 
		\begin{align*}
			|\mathbf{x}_{x_2}(t) -\mathbf{x}_{x_1}(t)|\leqslant K_0(n,T) \cdot |x_1-x_2|+ \widetilde K_0(n,T)\cdot |t_{x_1,nT}-t_{x_2,nT}|.
		\end{align*}
	     Since $t_{x,nT}$ is continuous with respect to $x$ for each $n\in \N$, we conclude that for any $t>0$, $\mathbf{x}_x(t)$ is continuous with respect to $x$.
	\end{proof}

\bibliographystyle{abbrv}

\bibliography{sing}

\end{document}